\documentclass{conm-p-l}
\usepackage{amssymb}
\usepackage{amsmath,amsthm}
\usepackage{graphicx}
\newtheorem{theorem}{Theorem}[section]

\theoremstyle{definition}

\theoremstyle{remark}
\newtheorem{remark}[theorem]{Remark}

\numberwithin{equation}{section}

\begin{document}


\title[Asymptotic Homotopical Complexity]{Asymptotic Homotopical Complexity of an Infinite Sequence of Dispersing $2D$ Billiards}


%    author one information
\author{Nandor Simanyi}
\address{1402 10th Avenue South \\
Birmingham AL 35294-1241}
% \curraddr{}
\email{simanyi@uab.edu}
% \thanks{}

% author two information
% \author{}
% \address{}
% \curraddr{}
% \email{}
% \thanks{}

% \subjclass[2000]{37D05}
%    The 2010 edition of the Mathematics Subject Classification is
%    now available.  If you are citing a classification from the
%    new scheme, use the following input coding instead.

\date{\today}

\begin{abstract}
  We investigate the large scale chaotic, topological structure of the
  trajectories of an infinite sequence of dispersing, hence ergodic,
  $2D$ billiards with the configuration space $Q_n=\mathbb{T}^2
  \setminus \bigcup_{i=0}^{n-1} D_i$, where the scatterers $D_i$
  ($i=0,1,\dots,n-1$) are disks of radius $r<<1$ centered at the
  points $(i/n, 0)$ mod $\mathbb{Z}^2$. We get effective lower and
  upper radial bounds for the rotation set $R$. Furthermore, we also
  prove the compactness of the admissible rotation set $AR$ and the
  fact that the rotation vectors $v$ corresponding to admissible periodic orbits
  form a dense subset of $AR$. We also obtain asymptotic lower and
  upper estimates for the sequence $h_{top}(n)$ of topological
  entropies and precise asymptotic formulas for the metric entropies
  $h_{\mu}(n,r)$.
\end{abstract}

\maketitle

\section{Introduction}
The concept of rotation number finds its origin in the study of the
average rotation around the circle $S^1$ per iteration, as classically
defined by H. Poincar\'e in the 1880's \cite{P(1952)}, when one
iterates an orientation-preserving circle homeomorphism $f:S^1
\rightarrow S^1$. This is equivalent to studying the average
displacement $(1/n)(F^n(x)-x)$ ($x \in \mathbb{R}$) for the iterates
$F^n$ of a lifting $F:\mathbb{R} \rightarrow \mathbb{R}$ of $f$ on the
universal covering space $\mathbb{R}$ of $S^1$. The study of fine
homotopical properties of geodesic lines on negatively curved, closed
surfaces goes back at least to Morse \cite{M(1924)}. As far as we
know, the first appearance of the concept of homological rotation
vectors (associated with flows on manifolds) was the paper of
Schwartzman \cite{Sch(1957)}, see also Boyland \cite{B(2000)} for
further references and a good survey of homotopical invariants
associated with geodesic flows.  Following an analogous pattern, in
\cite{BMS(2006)} we defined the (still commutative) rotation numbers
of a $2D$ billiard flow on the billiard table $\mathbb{T}^2 =
\mathbb{R}^2/\mathbb{Z}^2$ with one convex obstacle (scatterer)
$\mathbf{O}$ removed. Thus, the billiard table (configuration space)
of the model in \cite{BMS(2006)} was $\mathbf{Q} =
\mathbb{T}^2\setminus\mathbf{O}$.  Technically speaking, we considered
trajectory segments $\{x(t) | 0 \le t \le T\} \subset \mathbf{Q}$ of
the billiard flow, lifted them to the universal covering space
$\mathbb{R}^2$ of $ \mathbb{T}^2$ (not of the configuration space
$\mathbf{Q}$), and then systematically studied the rotation vectors as
limiting vectors of the average displacement
$(1/T)(\tilde{x}(T)-\tilde{x}(0)) \in \mathbb{R}^2$ of the lifted
orbit segments $\{\tilde{x}(t)|0 \le t \le T\}$ as $T \rightarrow
\infty$. These rotation vectors are still ``commutative'', for they
belong to the vector space $\mathbb{R}^2$.

Despite all the advantages of the homological (or ``commutative'')
rotation vectors (i. e. that they belong to a real vector space, and
this provides us with useful tools to construct actual trajectories
with prescribed rotational behaviour), in our current view the
``right'' lifting of the trajectory segments $\{x(t)|0 \le t \le T\}
\subset \mathbf{Q_n}$ is to lift these segments to the universal
covering space of $\mathbf{Q_n}$, not of $\mathbb{T}^2$. This, in turn,
causes a profound difference in the nature of the arising rotation
``numbers'', primarily because the fundamental group
$\pi_1(\mathbf{Q_n})$ of the configuration space $\mathbf{Q_n}$ is a
highly complex hyperbolic group.

After a bounded modification, trajectory segments $\{x(t)| 0 \le t \le T\} \subset
\mathbf{Q_n}$ give rise to closed loops $\gamma_T$ in $\mathbf{Q_n}$, thus
defining an element $g_T = [\gamma_T]$ in the fundamental group
$\pi_1(\mathbf{Q_n})$. The limiting behavior of $g_T$ as
$T \rightarrow \infty$ will be investigated, quite naturally, from two
viewpoints:

\begin{enumerate}
   \item The direction ``$e$'' is to be determined, in which the
     element $g_T$ escapes to infinity in the hyperbolic group
     $\pi_1(\mathbf{Q_n})$ or, equivalently, in its Cayley graph
     $\mathbf{G}$. All possible directions $e$ form the
     horizon or the so called ideal boundary
     $\text{Ends}(\pi_1(\mathbf{Q_n}))$ of the group $\pi_1(\mathbf{Q_n})$,
     see \cite{CP(1993)}.

   \item The average speed $s = \lim_{T \rightarrow \infty}
   (1/T)\text{dist}(g_T, 1)$ is to be determined, at which the element
   $g_T$ escapes to infinity, as $T \rightarrow \infty$. 
   These limits (or limits $\lim_{T_n \rightarrow \infty} 
   (1/T_n)\text{dist}(g_{T_n}, 1)$ for sequences of positive reals $T_n \nearrow \infty$) are nonnegative real numbers.
\end{enumerate}

The natural habitat for the two limiting data $(s,e)$ is the infinite cone

\begin{displaymath}
C=([0, \infty) \times \text{Ends}(\pi_1(\mathbf{Q_n}))/(\{0\} \times \text{Ends}(\pi_1(\mathbf{Q_n}))
\end{displaymath}
erected upon the set $\text{Ends}(\pi_1(\mathbf{Q_n}))$, the latter supplied
with the usual Cantor space topology. Since the homotopical ``rotation
vectors'' $(s,e) \in C$ (and the corresponding homotopical rotation
sets) are defined in terms of the non-commutative fundamental group
$\pi_1(\mathbf{Q_n})$, these notions will be justifiably
called homotopical or noncommutative rotation numbers and sets.

The rotation set arising from trajectories obtained by the arc-length
minimizing variational method will be the so called admissible
homotopical rotation set $AR \subset C$. The homotopical rotation set
$R$ defined without the restriction of admissibility will be denoted
by $R$. Plainly, $AR \subset R$ and these sets are closed subsets of
the cone $C$.

\bigskip

In this paper we study the large scale chaotic topological (actually,
homotopic) structure of the trajectories of an infinite sequence of
dispersing, hence ergodic, $2D$ billiards with a configuration space
$Q_n=\mathbb{R}^2/\mathbb{Z}^2\setminus\cup_{i=0}^{n-1} D_i =
\mathbb{T}^2\setminus\cup_{i=0}^{n-1} D_i$, where the scatterers $D_i$
($i=0, 1, \dots, n-1$) are disks of radius $r<<1$ centered at $(i/n,
0)$ modulo $\mathbb{Z}^2$. In Theorems 2.1 and 2.2 we provide lower and upper radial
estimates for the homotopical rotation set of the system. Namely, in Theorem 2.1
we prove
\[
R\subset B(0, 2\sqrt{2}),
\]
whereas Theorem 2.2 states that
\[
B\left(0, \frac{1}{\sqrt{5}}-\mathcal{O}(\frac{1}{n})\right)\subset AR.
\]

Furthermore, in Theorems 3.1 and 3.2 we also prove the compactness
and the convexity of the admissible rotation
set $AR$ and the property that the set of rotation vectors $v$
corresponding to periodic admissible orbits form a dense subset of
$AR$.  For all of the above notions, please see the first three pages
(up to the formulation of Theorem 2.2) of \cite{BMS(2006)}. In particular,
admissibility of a trajectory segment $S^{[0,T]}(x_0)=\left\{S^t(x_0)=\left(x(t),x'(t)\right)\big| 0\le t\le T\right\}$,
lifted to the universal covering space $\tilde{\mathbf{Q_n}}$ of the configuration space
$\mathbf{Q_n}$, means two things:

  (i) any two, consecutively visited obstacles (in the universal covering space
  of the configuration space, as always) are always different and the convex hull of them does not intersect
  any other obstacle in $\tilde{\mathbf{Q_n}}$;

  (ii) the obstacle visited at the $k$-th collision $k=2, 3, 4,\dots$ does not intersect
  the convex hull of the obstacles visited at the $k-1$-st and $k+1$-st collisions.

\medskip
  
For a fixed admissible sequence of obstacles $(\sigma_0, \sigma_1,\dots ,\sigma_n)$
from $\tilde{\mathbf{Q_n}}$, consider all piecewise smooth curves
\[
\left\{x(t)\big| 0\le t\le A\right\}
\]
in $\tilde{\mathbf{Q_n}}$ with the arc-length parametrization and with the following property:

(P) There are numbers $0=t_0<t_1<t_2<\dots <t_n=A$ such that $x(t_i)\in\partial\sigma_i$
($i=0,1,2,\dots, n$), and for all other values of $t$ the point $x(t)$ is in the interior
of $\tilde{\mathbf{Q_n}}$.

Elementary inspection shows that amongst the above curves there is a unique shortest one,
and the shortest curve is actually a billiard trajectory. Furthermore, for the shortest curve
the collisions at times $t_0=0$ and $t_n=A$ are perpendicular to the boundaries of the involved
scatterers $\sigma_0$ and $\sigma_n$, therefore the minimal curve $\left\{x(t)\big| 0\le t\le A\right\}$
is actually half of the periodic orbit $\left\{x(t)\big| 0\le t\le 2A\right\}$, where
$x(t)=x(2A-t)$ for $A\le t\le 2A$.

\medskip

The above observation is often called in the literature, in particular in \cite{BMS(2006)},
the variational principle, or variational method for the construction of admissible orbits.

\medskip
  
For the topological entropy $h_{top}(n)$ of the $n$-th system, in Theorem 4.1  we obtain
\[
\frac{1}{\sqrt{5}}\le \liminf_{n\to\infty}\frac{h_{top}(n)}{\log n} \le 
\limsup_{n\to\infty}\frac{h_{top}(n)}{\log n} \le 2\sqrt{2}.
\]

Finally, for the metric entropy $h_\mu (n,r)$ of the flow, in Section 5 we prove that for any fixed $n$,
\[
h_\mu(n,r)=\mathcal{O}(-r\log r)
\]
as $r\to 0$, and for $r=\mathcal{O}(1/n)$
\[
h_\mu(n,T)=\mathcal{O}(\log(n)),
\]
as $n\to \infty$ for the Poincare section map $T$.

\bigskip

\section{Radial Size Estimates for the Admissible Rotation Set}

\bigskip

As described in Section 2 of \cite{GS(2011)} (see page 4 in that paper, in particular Figures 1 and 2),
the free generators of the fundamental group
$\pi_1(Q_n)\cong F_{n+1}$ are $a,\, b_1,\, \dots,\, b_n$, so that each of these generators correspond
to a specific wall crossing of a trajectory: The generator $a$ corresponds to a trajectory
crossing the wall
$$
\{0\}\times [r,\, 1-r]
$$
(modulo $\mathbb{Z}^2$) in the positive direction, whereas the generator $b_i$ means that the
trajectory crosses the wall
$$
\left[\frac{i-1}{n}+r,\, \frac{i}{n}-r\right]\times\{0\}
$$
(taken again modulo $\mathbb{Z}^2$) in the positive direction. Needless to say, crossing these walls in the
opposite, negative directions means the inverses of the above mentioned elements of $\pi_1(Q_n)$.

\begin{figure}[h]
   \centerline{
   \includegraphics[width=0.6\textwidth, height=0.3\textheight]{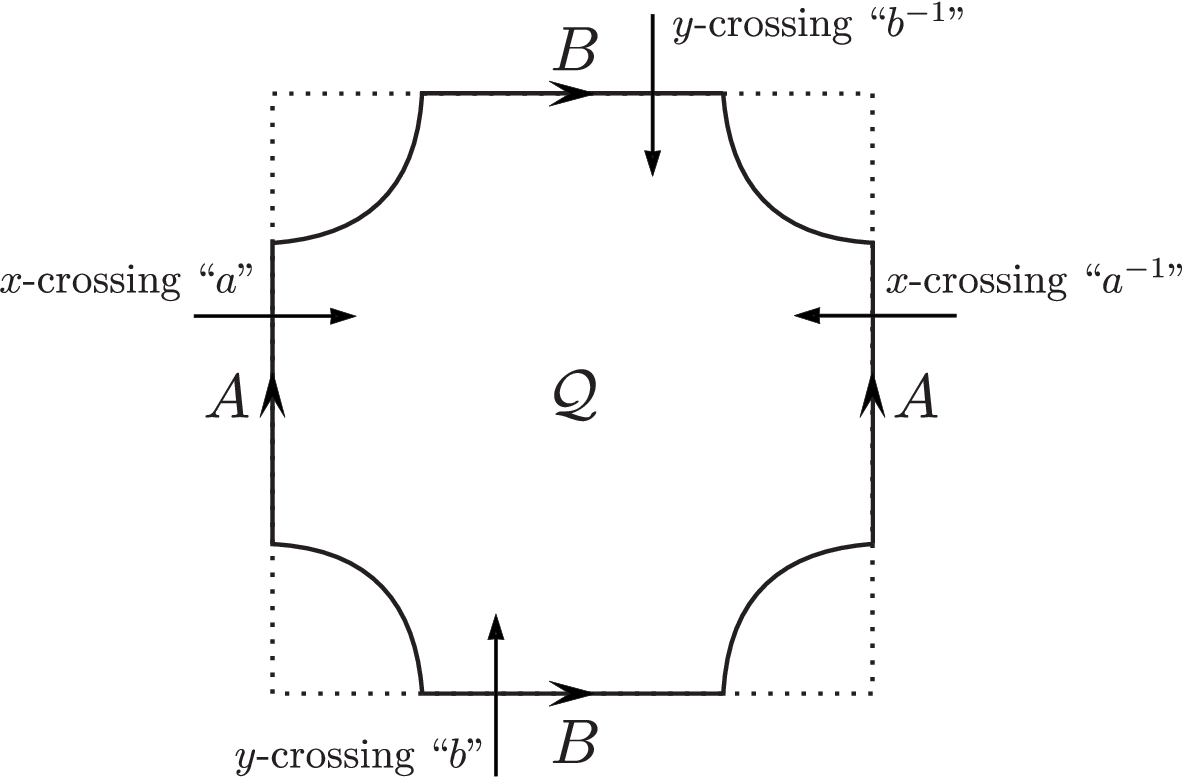}
   }
   \label{Figure 1}
   \caption{}
\end{figure}

\begin{figure}[h]
   \centerline{
   \includegraphics[width=0.6\textwidth, height=0.25\textheight]{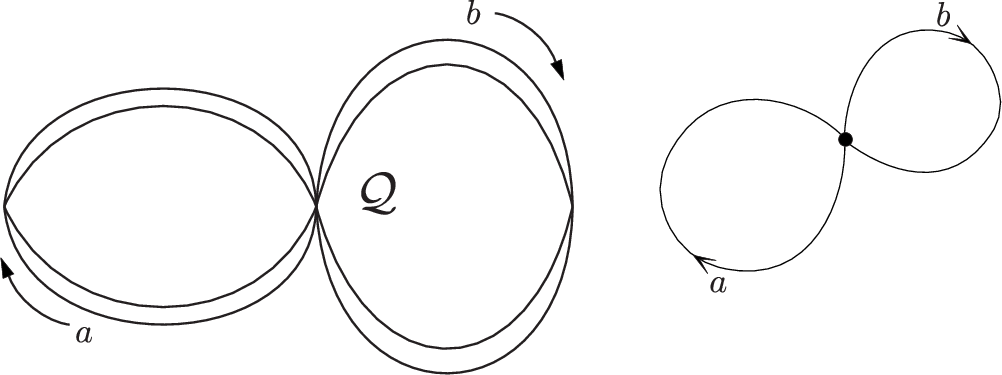}
   }
   \label{Figure 2}
   \caption{}
\end{figure}

First we prove the upper radial size estimate:

\medskip

\begin{theorem} $R\subset B(0, 2\sqrt{2})$
\end{theorem}

\medskip

\begin{proof}  
Let us investigate the symbolic wall-crossing sequence, or associated word, $W$ of a trajectory segment $S^{[0,T]}x_0$
of the billiard flow. We write $W$ as the shortest product of generators $a, b_1, b_2, \dots ,b_n$ and their inverses,
i.e. only the exponents $1$ and $-1$ are permitted. In such a writing, a maximal subsequence of consecutive symbols
$a$ or $a^{-1}$ of $W$ will be called an $a$-block $B^a$, whereas a maximal subsequence of consecutive symbols
$b_j$ or $b_j^{-1}$ ($j=1, 2,\dots, n$) of $W$ will be called a $b$-block $B^b$. To be clear, in a $b$-block, distinct
symbols $b_i$ and $b_j$ ($i\ne j$) and their inverses are allowed to be elements of the same $b$-block.
For just simplifying the notations,
we assume that $W$ begins with an $a$-block $B_1^a$, and it ends with a $b$-block $B_s^b$, so that $W$ is of the form
\[
W=B_1^a B_1^b B_2^a B_2^b \dots B_s^a B_s^b.
\]
At the end of proving the upper radial size estimate $2\sqrt{2}$ for such words, we will be clearly pointing out the minor
differences in the argument to cover the other cases for the word $W$.

Let us denote the total number of the symbols $a$ and $a^{-1}$ in $W$ by $k$, whereas the total number of the symbols
$b_j^{\pm 1}$ ($j=1, 2, \dots ,n$) is denoted by $m$. Clearly, $s\le \text{min}\{k, m\}$.

Observe that for our trajectory segment
\[
S^{[0,T]}x_0=\left\{\left(x(t), x'(t)\right)\big|\; 0\le t\le T\right\}
\]
the estimate
\begin{equation}\label{integral}
\int_0^T \left|x'_1(t)\right|dt \ge k-s
\end{equation}
holds true. We note that here, and everywhere else in this paper, $(x(t), x'(t))$
denotes the phase point at time $t$, where $x(t)$ is the position of the moving particle, already
lifted to the universal covering space $\tilde{Q_n}$ of the configuration space
$Q_n=\mathbb{R}^2/\mathbb{Z}^2\setminus\cup_{i=0}^{n-1} D_i=\mathbb{T}^2\setminus\cup_{i=0}^{n-1} D_i$,
and $x'(t)=\dfrac{d}{dt}x(t)$ is its time derivative. Furthermore, $x_1(t)$ and $x_2(t)$ denote the
first and second coordinates of $x(t)$, and similar notational convention holds true for the time
derivative $x'(t)$.

Indeed, for each $a$-block $B_j^a$ the horizontal position changes at least by the amount $|B_j^a|-1$, where
$|B_j^a|$ is the number of symbols in $B_j^a$. Taking the sum of all these lower estimates, we obtain \ref{integral}.

The situation with $b$-blocks $B_j^b$ and the corresponding vertical move
\[
\int_0^T \left|x'_2(t)\right|dt
\]
is somewhat different. Namely, we consider not only the block $B_j^b$, but also the subsequent two blocks $B_{j+1}^a$
and $B_{j+1}^b$. Elementary inspection shows that, after having at least the amount of $|B_j^b|-1$ of vertical motion
with $B_j^b$, in order to switch from $B_j^b$ to $B_{j+1}^b$ through the block $B_{j+1}^a$, the orbit has to make at least
a unit of vertical motion betwen $B_j^b$ and $B_{j+1}^b$. (More precisely, this motion is at least $2-\mathcal{O}(1/n)$,
if one of $B_j^b$ and $B_{j+1}^b$ corresponds to an upward motion and the other one to a downward motion.)

Thus, by taking the sum of all these lower estimates, one obtains the lower bound 
\begin{equation}\label{integral2}
\int_0^T \left|x'_2(t)\right|dt \ge m-1.
\end{equation}
Taking the sum of \ref{integral} and \ref{integral2}, one ends up with the estimates
\begin{equation}\label{integral3}
m+k-s-1 \le\int_0^T \left||x'(t)\right||_1dt \le \sqrt{2}T,
\end{equation}
which, in turn, implies that
\[
\text{max}\{m,k\}-1\le \sqrt{2}T,
\]
thus
\begin{equation}\label{wupper}
|W|=m+k\le 2\sqrt{2}T+2.
\end{equation}
Dividing \ref{wupper} by $T$ and passing to the limit $T\to \infty$ yields
\[
R\subset B(0, 2\sqrt{2}).
\]

As said earlier, here we review the minor changes in the proof that are necessary to make, if the first block of $W$
is not an $a$-block, or the last block of it is not a $b$-block. Then, by simply truncating the unwanted block(s) at the
two ends of $W$ (thus also truncating the trajectory segment under investigation), we can reduce the study to the above case.
Note that, in this way, we make a bounded error in the right-hand-side of \ref{wupper}, but that error vanishes after division
by $T$ and passing to the limit $T\to \infty$. This completes the proof of the theorem.
\end{proof}

\medskip

\begin{theorem}
The lower radial estimate
\[
B\left(0, \frac{1}{\sqrt{5}}-\mathcal{O}(\frac{1}{n})\right)\subset AR
\]
holds for the admissible rotation set $AR$.
\end{theorem}

\medskip

\begin{proof}
For any given infinite word $W_\infty=w_0w_1w_2\dots$ (where each $w_j$ either a generator or the inverse of a generator, and
$W_\infty$ is in the shortest form) we are going to construct an infinite sequence
\[
\left(S^{[0, T_m]}x_m\right)_{m=1}^\infty 
\]
of admissible
orbit segements with $\lim_{m\to \infty}T_m=\infty$ and such that the symbolic wall crossing sequence of $S^{[0, T_m]}x_m$ is
$w_0w_1\dots w_m$. In this process of constructing admissible orbit segments we use the arc-length minimizing variational
method, described earlier. We will try to economize time and make $T_m$ as small as possible, thus making
the escape speed
\[
\liminf_{m\to\infty}\frac{m}{T_m}
\]
as large as possible. We want to prove that
\begin{equation}\label{liminfbig}
\liminf_{m\to\infty}\frac{m}{T_m} \ge \frac{1}{\sqrt{5}}-\mathcal{O}(\frac{1}{n})
\end{equation}
can be achieved. Then, by ``slowing down'' the speed of escape to infinity by inserting ``idle runs'', just as in the
proofs of Theorems 3.1 and 3.3 in \cite{MS(2017)},
one gets that any speed $s$, $0\le s\le \frac{1}{\sqrt{5}}-\mathcal{O}(\frac{1}{n})$, is achievable while escaping to
infinity in the given (but arbitrary!) direction of $W_\infty$, thus proving the desired set inequality
\[
B\left(0, \frac{1}{\sqrt{5}}-\mathcal{O}(\frac{1}{n})\right)\subset AR.
\]

In order to prove \ref{liminfbig}, it is enough to show that, when constructing longer and longer orbit segments with the symbolic
sequence $w_0w_1\dots w_m$, in order to construct the passage $w_j\to w_{j+1}$, it is sufficient to spend at most the time
$\sqrt{5}+\mathcal{O}(1/n)$. We consider all possible, mutually non-isomorphic passages $w_j\to w_{j+1}$ by excluding geometrically
symmetric cases and cases equivalent after time reversal. After this reduction, here is the list of the remaining four cases
for the passages:

\subsection{Case 1. Passage $ab_i$}
Direct inspection shows that here the worst case scenario is when $i=n$ and the admissible orbit
(under construction) enters the unit fundamental cell at the upper end of the left vertical
$a$-wall. Then, the most efficient admissible construction bounces back at the disk with index
$\left\lfloor\dfrac{n}{2}\right\rfloor$ at the lower edge of the unit cell, before leaving it
upwards through the wall $b_n$. The time spent to do this is at most $\sqrt{5}+\mathcal{O}(1/n^2)$.

\medskip

\begin{remark} The possible error $\mathcal{O}(1/n^2)$ may be caused by the occurence of an odd value
of $n$, $n=2q+1$, and having to select $D_q$ or $D_{q+1}$ at the bottom side of the unit cell.
\end{remark}

\medskip

\subsection{Case 2. Passage $b_ib_j^{-1}$ ($i\ne j$)}

\bigskip

Here again, direct inspection shows that the least favorable (i.e. the most time consuming) situation is
when $|i-j|$ is maximal, i.e. $i=1$ and $j=n$. (Or the other way around.) Then the most time efficient construction
of the admissible trajectory in the unit fundamental cell $[0,1]\times [0,1]$ is to have it bounce back at the disk
with index $\lfloor n/2\rfloor$ at the ceiling of the unit cell. This takes time $\sqrt{5}+\mathcal{O}(1/n)$.

\begin{remark} As opposed to the previous case, the reason why the error is of order $\mathcal{O}(1/n)$
(and not $\mathcal{O}(1/n^2)$) is that, in the recursive construction of the passages of the admissible
trajectory, it can happen that the incoming wall crossing $b_1$ takes place while bouncing back at the
scatterer $D_1$. In order to facilitate the admissible passage to $D_{\lfloor n/2\rfloor}$ at the ceiling of the
unit fundamental cell, the orbit needs to visit the scatterer $D_0$ first, before shooting towards
$D_{\lfloor n/2\rfloor}$ at the ceiling. This takes an extra time of $\mathcal{O}(1/n)$.
\end{remark}

\subsection{Case 3. Passage $aa$}
By excluding symmetric cases, we may assume that the orbit enters the unit fundamental cell by bouncing back at $D_0$ at the top of the
left wall ``$a$''. Then we can have it exit the fundamental unit cell through the right wall ``$a$'' by bouncing back from
$D_0=D_n$ at the bottom of this wall. This takes time $\sqrt{2}$.

\subsection{Case 4. Passage $b_ib_j$}
Here again the least favorable case is when $|i-j|$ is maximal, i.e. $i=1$ and $j=n$ (or, vice versa). Then an admissible trajectory
construction in the unit fundamental cell can diagonally traverse the cell, similarly to Case 3, thus spending time not more
than $\sqrt{2}+\mathcal{O}(1/n)$ in the cell. The $\mathcal{O}(1/n)$ error term should be included here for the same reason
as in Case 2, but this does not matter, since $\sqrt{2}+\mathcal{O}(1/n)<\sqrt{5}$.

Summarizing all the above, when constructing the admissible orbit segments $S^{[0,T_m]}x_m$ with the symbolic sequence
$W_m=w_0w_1\dots w_m$ (being the $m$-th truncation of the given, infinite admissible word $W_\infty=w_0w_1\dots$)
we can achieve that
\[
\liminf_{m\to\infty}\frac{m}{T_m}\ge\frac{1}{\sqrt{5}+\mathcal{O}(1/n)}=\frac{1}{\sqrt{5}}-\mathcal{O}(1/n).
\]
This completes the proof of the theorem.
\end{proof}

\section{Geometric Properties of the Admissible Rotation set $AR$}

\bigskip

\begin{theorem}
The set $AR$ is a convex, compact subset of the infinite cone $C$ erected upon the topological Cantor set
$\text{Ends}(\pi_1(Q_n))$.

(Here the set $\text{Ends}(\pi_1(Q_n))$ denotes the set of all ends, i.e. the horizon of, the hyperbolic
group $\pi_1(Q_n)$, see (1) on page 2 of \cite{MS(2017)}.)
\end{theorem}

\begin{proof}
First we note that the convexity of a compact subset $K$ of the cone
$C$ is equivalent to the star shaped property of $K$. This is a
consequence of the fact that the basis $\text{Ends}(\pi_1(Q_n))$ of the
cone is totally disconnected.  Indeed, the shortest path (geodesic
line) connecting two points $(s_1, e_1)$ and $(s_2, e_2)$ of $C$
($s_1>0$, $s_2>0$, $e_1\ne e_2$) is the curve $\gamma$ that first connects
$(s_1, e_1)$ with the vertex $(0, e_1)$ by a linear change in the
radial speed coordinate $s$, then it connects the vertex with $(s_2, e_2)$
by also a linear change in the $s$ coordinate.

We follow the ideas of the proof of Theorem 3.3 of \cite{MS(2017)}. Indeed,
the cone $C$ is a totally disconnected, Cantor set-type family of
infinite rays that are glued together at their common endpoint, the
vertex of the cone. Therefore, the convexity of $AR$ means that for
any $(s,e)\in AR$ and for any $t$ with $0\le t\le s$ we have $(t,e)\in
AR$. However, this immediately follows from our construction, since we
can always insert a suitable amount of idle runs into an admissible
orbit segment to be constructed, hence slowing it down to the
asymptotic speed $t$, as required.
\end{proof}

\begin{theorem} The rotation vectors $(s,e)\in AR$ that correspond to periodic admissible trajectories
form a dense subset of $AR$.
\end{theorem}

\begin{proof} We adopt the main ideas of the proof of Theorem 3.4 of \cite{MS(2017)}. Consider an arbitrary
  trajectory segment
  \[
  S^{[0,T]}x_0=\left\{S^t x_0=(x_t,x'_t)\big|\; 0\le t\le T\right\}
  \]
  with $x(0), x(T)\in \partial Q_n$ that bounces back from the scatterers with centers at
  $C_0, \dots ,C_n$ (so that $\Vert C_i-x(t_i)\Vert=r$ for $0=t_0<t_1<t_2<\dots <t_n=T)$ with an admissible symbolic
  sequence $(C_0, C_1, \dots ,C_n)$. We want to find a periodic, admissible trajectory near $S^{[0,T]}x_0$
  such that its rotation vector is close to the rotation vector of $S^{[0,T]}x_0$.

  A direct inspection shows that the symbolic sequence $(C_0, C_1, \dots ,C_n)$ has a bounded
  extension to a longer admissible sequence $(C_0, C_1, \dots ,C_m)$, i.e. $m>n$ and $m-n$ has
  an upper bound $K$, independent of $S^{[0,T]}x_0$ and $n$, such that
  \begin{enumerate}
  \item[$(1)$] $C_m=C_0+(a,b)$, \, $a, b\in\mathbb{Z}$,
  \item[$(2)$] $\left(C_{m-1}, C_m, C_1+(a,b)\right)$ is admissible.
  \end{enumerate}
  Then, by anchoring the trajectory segment at $C_0$ and $C_m$ with the configuration points
  $x(0)$ (with $\Vert x(0)-C_0\Vert=r$) and $x(T)=x(0)+(a,b)$, and performing the usual length minimization,
  we obtain a periodic, admissible orbit $S^{[0,T']}x_0'$ whose rotation vector is close to the rotation
  vector of $S^{[0,T]}x_0$.
\end{proof}

\section{Topological Entropy}

\bigskip

Here we prove

\begin{theorem}
\[
\frac{1}{\sqrt{5}}\le\liminf_{n\to\infty}\frac{h_{top}(n)}{\log(n)}\le
\limsup_{n\to\infty}\frac{h_{top}(n)}{\log(n)}\le 2\sqrt{2}.
\]
\end{theorem}

\begin{proof} It turns out that the estimates of this theorem are corollaries of Theorems 2.1 and 2.2.
  Indeed, following the ideas of Section 5 of \cite{MS(2017)}, we construct a generating partition
  $\mathcal{P}$ for the topological entropy of the billiard flow, as follows:

  We define the following subsets of the full phase space $M=M_n$ of the studied system with $n$
  scatterers:

  \begin{equation}\label{dplus}
    D_k^+=\left\{(q,v)\in M\big|\; \frac{k}{n}<q_1<\frac{k+1}{n},\, \{q_2\}<\epsilon_0\right\},
    \end{equation}

  \begin{equation}\label{dminus}
    D_k^-=\left\{(q,v)\in M\big|\; \frac{k}{n}<q_1<\frac{k+1}{n},\, 1-\epsilon_0<\{q_2\}\right\}
    \end{equation}
for some $0<\epsilon_0<<r$ (i.e. $\epsilon_0=o(r)$ is small ordo of $r$), $k=0, 1, \dots, n-1$,

\begin{equation}\label{splus}
    S^+=\left\{(q,v)\in M\big|\; \{q_1\}<\epsilon_0\right\},
    \end{equation}

\begin{equation}\label{sminus}
    S^-=\left\{(q,v)\in M\big|\; \{q_1\}>1-\epsilon_0\right\},
    \end{equation}
where $\{x\}=x-\lfloor x \rfloor$ denotes the fractional part of $x$. By definition, the partition $\mathcal{P}$
consists of all sets in \ref{dplus}, \ref{dminus}, \ref{splus}, \ref{sminus} and the rest of the phase space

\begin{equation}\label{slab}
  B=M\setminus\left(\bigcup_{k=0}^{n-1} D_k^+\cup\bigcup_{k=0}^{n-1} D_k^- \cup S^+\cup S^-\right).
  \end{equation}
Just like in Section 5 of \cite{MS(2017)}, an easy geometric inspection shows that the $\mathcal{P}$-itinerary
$(P_1, P_2,\dots ,P_k)$ of a trajectory segment $S^{[a,b]}x_0$ uniquely determines the homotopy type of $S^{[a,b]}x_0$.

Similarly, if one considers the discrete-time dynamics, i.e. the collision map, then the full $\mathcal{P}$-itinerary
of a phase point uniquely determines the phase point, that is, the partition $\mathcal{P}$ is a
generating partition for the topological entropy.

The lower radial estimate of Theorem 2.2 allows us to construct admissible orbit segments $S^{[0,T]}x$,
with a given time span $T$, with at least as many as

\begin{equation}\label{hatvany}
  (2n+1)^{(5^{-1/2}-\mathcal{O}(1/n))T}
\end{equation}
different $\mathcal{P}$-itineraries. By taking the natural logarithm of \ref{hatvany}, dividing by $T$, and passing to the
limit as $T\to\infty$, one concludes that
\[
h_{top}(n)\ge (5^{-1/2}-\mathcal{O}(1/n))\log(2n+1).
\]
Finally, dividing by $\log(n)$, then passing to the limit inferior as $n\to\infty$, we ontain the lower bound of Theorem 4.1.

On the other hand, Theorem 2.1 shows that the orbit segments $S^{[0,T]}x$, with a given time span $T$, cannot form
more than
\[
(2n+1)^{2\sqrt{2}T}
\]
different $\mathcal{P}$-itineraries. Again, by taking natural logarithm, dividing by $T$, passing first to the limit
$T\to\infty$, then to the limit superior as $n\to\infty$, we obtain the upper bound of Theorem 4.1.
\end{proof}

\bigskip

\section{Concluding Remarks}

\bigskip

In this brief section we make a quick comparison between the topological entropy $h_{top}(n)$ of the flow
and its metric entropy $h_\mu (n)=h_\mu(n,r)$, where $\mu=\mu_n$ is the unique, absolutely continuous invariant measure of
the billiard flow. As it turns out, if the radius $r=r(n)$ of the disks is a sufficiently fast decreasing
function of $n$, than $h_\mu (n)$ is dramatically smaller than $h_{top}(n)$. Indeed, first fix the value of
$n$, and study the behaviour of the metreic entropy $h_\mu(n,r)$ of the flow as a function of $r$, when
$r\to 0$. As it follows from the Marklof-Strömbergsson theory of the distribution of the free path length
in a dilute Lorentz gas \cite{Ma-St(2010)}, the expected value of the free path length is $\mathcal{O}(1/r)$.
After a typical collision with a scatterer of radius $r$, an expanding (convex) local orthogonal manifold
collects a curvature of the order of $\mathcal{O}(1/r)$, and undergoes an expansion of the order
$\mathcal{O}(1/r^2)$ until the next collision. Thus the positive Lyapunov exponent will be of the
order $\mathcal{O}(-r\log(r))$, therefore $h_\mu(n,r)=\mathcal{O}(-r\log(r))$ for any fixed $n$, while
$r\to 0$. This immediately implies that, if $r=r(n)$ is a sufficiently fast decreasing function of $n$,
then the ratio $\dfrac{h_\mu(n)}{h_{top}(n)}$ can be made to converge to zero as fast as we wish,
as $n\to \infty$.

On the other hand, if $r=\mathcal{O}(1/n)$ and we consider the billiard map $T$ instead of the
flow $\{S^t\}$, then the metric entropy $h_\mu(n,T)$ will be of the order $\mathcal{O}(\log(n))$,
i.e. of the order of the topological entropy of the map. Indeed, the iterate $T^k$ produces asymptotically
the same amount of dilation of the local unstable manifolds (as $k\to \infty$) as the iterate
$S^{k/r}\approx S^{kn}$ of the flow so, after rescaling, we gain an extra multiplying factor $n$ in the
metric entropy, i.e. $h_\mu(n,T)=\mathcal{O}(-nr\log(r))=\mathcal{O}(\log(n))$.

\bigskip

\bibliographystyle{amsplain}

\end{document}